\documentclass [a4,11pt]{amsart}

\usepackage{amsmath,amssymb,amsfonts,enumerate,amsthm, amscd,}
\usepackage[all]{xy}
\usepackage{graphicx}

\renewcommand{\dim}{\mbox{dim}\,}

\newtheorem{thm}{Theorem}[section]
\newtheorem{cor}[thm]{Corollary}
\newtheorem{lem}[thm]{Lemma}
\newtheorem{prop}[thm]{Proposition}
\newtheorem{defn}[thm]{Definition}

\newtheorem{exam}[thm]{Example}
\newtheorem{rem}[thm]{Remark}

\def\proof{{\parindent0pt {\bf Proof.\ }}}
\newcommand{\field}[1]{\mathbb{#1}}

\newcommand{\Q }{\field{Q}}
\newcommand{\Z }{\field{Z}}

\theoremstyle{definition}

\theoremstyle{remark}

\theoremstyle{Definition and Notation}

\begin{document}

\bibliographystyle{amsplain}

%\date{20\03\2011}

\title{Commutative rings with invertible-radical factorization}

\author{Malik Tusif Ahmed}
\address{Malik Tusif Ahmed\\Abdus Salam School of Mathematical Sciences, Government College University  Lahore, Pakistan
$$ E-mail\ address:\ tusif.ahmad92@gmail.com $$}

\author{Najib Mahdou}
\address{Najib Mahdou\\Laboratory of Algebra, Functional Analysis and Applications, \\Department of Mathematics, Faculty of Science and Technology of Fez, Box 2202,
University S.M. Ben Abdellah Fez, Morocco.
$$E-mail\ address:\ mahdou@hotmail.com$$}

\author{Youssef Zahir}
\address{ZAHIR Youssef\\Department of Mathematics, Faculty of Science and Technology of Fez, Box 2202,
University S.M. Ben Abdellah Fez, Morocco.
$$ E-mail\ address:\ youssef.zahir@usmba.ac.ma$$}

\thanks{The first named author is supported by ASSMS, GC University Lahore under a Postdoc fellowship.}


%Subject classification numbers        13B99           13B21

\begin{abstract}
In this paper, we study the classes of rings in which every proper (regular) ideal can be factored as an invertible ideal times a nonempty product of proper radical ideals. More precisely, we investigate the stability of these properties under homomorphic image and their transfer to various contexts of constructions such as direct product, trivial ring extension and amalgamated duplication of a ring along an ideal. 
%Along the sections, we study some basic results of the two properties to prove the relationship with some existing related classes of rings. 
Our results generate examples that enrich the current literature with new and original families of rings satisfying these properties.
\end{abstract}

\maketitle

%%%%%%%%%%%%%%%%%%%%%%%%%%%%%%%%%%%%%%%%%%%%%%%%%%%%%%%%%

\bigskip
 %%%%%%%%%%%%%%%%%%%%%%%%%%%%%%%%%%%%%%

%%%%%%%%%%%%%%%%%%%%%%%%%%%%%%%%%%%%%%%%%%%%%%%%%%%%%%%%
%%%%%%%%%%%%%%%%%%%%%%%%%%%%%%%%%%%%%%%%%%%%%%%%%%%%%%%%
%%%%%%%%%%%%%%%%%%%%%%%%%%%%%%%%%%%%%%%%%%%%%%%%%%%%%%%%

%%%%%%%%%%%%%%%%%%%%%%%%%%%%%%%%%%%%%%%%%%%%%%%%%%%%%%%%%
%%%INTRODUCTION%%%%%%%%%%%%%%%%%%%%%%%%%%%%%%%%%%%%%%%%%%

\section{Introduction}
The opening part is devoted to some standard background material. Throughout this paper all rings are commutative with a nonzero unit and all modules are unitary. For a ring $A$ and an $A$-module $E$, we shall use $Z(A)$ to denote the set of zero-divisors of $A$ and refer to an element that is not contained in $Z(A)$ as being regular. Also, we will denote by  $Z(E)$ the set of zero-divisors on $E$. A regular ideal is an ideal that contains at least one regular element. For any undefined terminology see \cite{RG} and \cite{Huckaba}.\par
In \cite{MT}, T. Dumitrescu and the first named author of the current note introduced and studied the notion of  an ISP-domain, that is, integral domain whose ideals can be factored as an invertible ideal times a nonempty product of proper radical ideals (this terminology comes from ``invertible semiprime ideal"). They proved that if $A$ is an ISP-domain, then any factor domain of $A$ and any (flat) overring of $A$ are also ISP-domains. They also showed that if $A$ is an ISP-domain, then $A$ is strongly discrete Pr\"ufer (i.e. a Pr\"ufer domain having no idempotent prime ideal except the zero ideal) and every nonzero prime ideal of $A$ is contained in a unique maximal ideal. The relevant background on the underlying domain-theoretic properties and their generalizations is presented in the following paragraph.\par

In \cite{VY}, N. Vaughan and R. Yeagy introduced the class of SP-domains i.e. domains whose proper ideals are product of radical ideals (see also \cite{O1}). They proved that an SP-domain is almost Dedekind (see also \cite{A}). In \cite{SP}, T. Dumitrescu and the first named author of the current note generalized the study of the SP property to the context of arbitrary rings in two ways as follows. A ring $A$ is called an SP-ring (resp. SSP-ring) if each  proper regular ideal (resp. ideal) is a product of radical ideals.
In \cite{O2,O3}, Olberding introduced and studied the following class of rings. A ring $A$ is called a ZPUI-ring, if every proper ideal of $A$ can be factored as an invertible ideal times a nonempty product of prime ideals.

The following construction was introduced by Nagata \cite[p. 2]{Nagata}. The trivial ring extension of a ring $A$ by an $A$-module $E$ (also called the idealization of $E$ over $A$) is the ring $A\propto E$  whose underlying group is $A\times E$ with multiplication given by $(a,e)(b,f)= (ab,af+be)$. For more details on trivial ring extensions, we refer the reader to Glaz's and Huckaba's respective books \cite{Gl,Huckaba}. We also refer D. D. Anderson and M. Winders relatively recent and comprehensive survey paper \cite{Anderson2}. These have proven to be useful in solving many open problems
and conjectures for various contexts in (commutative and non-commutative) ring theory,
see for instance \cite{MM,Anderson2,BKM,Huckaba,KM}.

The amalgamated duplication of a ring $A$ along an ideal $I$, introduced and studied by D'Anna and denoted by $A\bowtie I$, is the following subring of $A\times A$ (endowed with the usual
componentwise operations):
$$A\bowtie I=\{(a,a+i)~|~ a\in A \text{ and } i\in I\}.$$
Note that if $I^{2}=0$, then this construction $A\bowtie I$ coincides with the trivial ring extension $A\propto I$. One main difference between $A\bowtie I$ and $A\propto I$ is that the former ring can be a reduced ring, for example, it is always reduced if $A$ is an integral domain. Motivations and additional applications of the amalgamated duplication are discussed in more detail in \cite{J,D'anna1,Fontana1}.\par

In this note, we extend the ISP-domain concept to rings with zero-divisors in two different ways. Section $2$ is devoted to the study of the first class of rings in which every proper regular ideal can be factored as an invertible ideal times a nonempty product of proper radical ideals (called by us ISP-rings). Also, we investigate the stability of this property under regular localization and homomorphic image, and its transfer to various contexts of constructions such as direct product (Proposition~\ref{Pro}), trivial ring extension (Theorem~\ref{exten}) and amalgamated duplication of a ring along an ideal (Theorem~\ref{dup}).

Section $3$ deals with the study of ISP-rings which are also Marot (i.e. rings whose regular ideals are generated by regular elements, see \cite[p. 31]{Huckaba}). Among other useful results we also prove that any ISP-ring $A$ whose regular prime ideals are maximal, is an $N$-ring, that is, $A_{(M)}$ is a discrete rank one Manis valuation ring for each regular maximal ideal $M$ of $A$ (Theorem~\ref{ispN}). This result extends \cite[Corollary 1]{MT}. Here $A_{(M)}$ is the regular localization of $A$ at $M,$ that is, the fraction ring $A_S$ where $S=Reg(A)\cap (A-M)$. As a consequence of Theorem~\ref{ispN}, we prove that any regular-Noetherian ISP-ring is Dedekind, that is, a ring whose regular ideals are products of prime ideals (Corollary~\ref{rnisp}).

 %and prove the following results. If $A$ is an ISP-ring and $I$ a proper invertible radical ideal of $A$, then $A$ is zero-dimensional, that is, $\dim(A/I)=0$ (Lemma \ref{zero-dim}). Moreover, any ISP-ring whose regular prime ideals are maximal, is an $N$-ring, that is, $A_{(M)}$ is a discrete rank one Manis valuation ring for each maximal regular ideal $M$ of $A$ (Theorem \ref{ispN}). This result extends \cite[Corollary 1]{MT}. Here $A_{(M)}$ is the regular localization of $A$ at $M$, i.e. the fraction ring $A_S$ where $S=Reg(A)\cap (A-M)$. As a consequence of Theorem~\ref{ispN}, we prove that any regular-Noetherian ISP-ring is Dedekind ring (Corollary \ref{rnisp}).

Section $4$ is devoted to the class of rings in which every proper ideal can be factored as an  invertible ideal times a nonempty product of proper radical ideals (called by us strongly ISP-rings). Similar as for the ISP-ring case, we investigate the stability of this property under factor with prime ideal, fraction and finite direct product ring formations (Proposition~\ref{sisp}). Further, we also study its transfer to various contexts of constructions such as trivial ring extension (Proposition~\ref{car}) and amalgamated duplication of a ring along an ideal (Theorem~\ref{dupli}). Besides this and other useful results we also prove that any strongly ISP-ring $A$ whose nonzero prime ideals are maximal, is an almost multiplication ring, that is, for every prime ideal $P$, the localization $A_P$ is a discrete rank one valuation domain or a special primary ring (Theorem~\ref{sispamr}). This is another extension of \cite[Corollary 4]{MT}. As a consequence of Theorem~\ref{sispamr}, we prove that any Noetherian strongly ISP-ring is ZPI, that is, a ring whose proper ideals are products of prime ideals (Corollary~\ref{nsisp}).

% Namely, we prove that, if $A$ is a strongly ISP-ring, then any fraction ring (resp. finite direct product) of a strongly ISP-ring is a strongly ISP-ring (Proposition \ref{sisp}). Also, by using the trivial ring extension, we give an example showing that the above classes (ISP-rings and strongly ISP-rings) of rings are not equivalent (Example \ref{example}). If $A$ is strongly ISP-ring, then any invertible radical proper ideal of $A$ is zero-dimensional (Lemma \ref{irsisp}). If $A$ is a zero-dimensional local strongly ISP-ring, then $A$ is a special primary ring, i.e. a local ring whose proper ideals are powers of its maximal ideal (Proposition \ref{spr}). In Theorem \ref{sispamr} we prove that a Strongly ISP-ring $A$ whose nonzero prime ideals are maximal, is an almost multiplication ring, i.e. for every prime ideal $P$, the localization $A_P$ is a discrete rank one valuation domain or a special primary ring. This is another extension of \cite[Corollary 4]{MT}. In Corollary \ref{nsisp} we  prove that Noetherian strongly ISP-rings are ZPI-rings, that is, the rings whose proper ideals are product of prime ideals. At the end of this section, we study transfer of ``strongly ISP-ring" property under trivial ring extension (Theorem~\ref{car}) and amalgamated duplication of a ring along an ideal (Theorem~\ref{dupli}). 

%For the main transfer results in this paper, see Theorems \ref{exten}, \ref{dup}, \ref{ispN}, \ref{sispamr} and \ref{dupli}. 
As we proceed to study the above-mentioned classes of rings, the reader may find it helpful to keep in mind the implications noted in the following figure.
\\
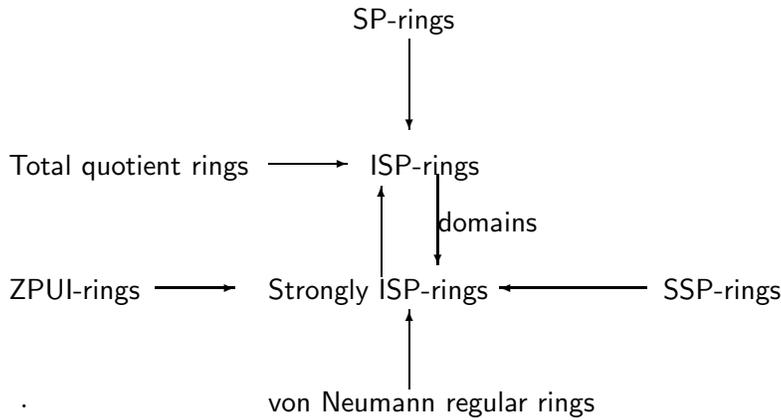
\begin{center}
\setlength{\unitlength}{1.5mm}
\begin{picture}(140,35)(18,-2)
\put(17,15){\textsf{Total quotient rings}}
\put(49,15){\textsf{ISP-rings}}
\put(55,10){\textsf{domains}}
\put(47.5,28){\textsf{SP-rings}}
\put(17,4){\textsf{ZPUI-rings}}
\put(40,4){\textsf{Strongly ISP-rings}}

\put(75,4){\textsf{SSP-rings}}
\put(40,-6){\textsf{von Neumann regular rings}}

\put(40,16){\vector(3,0){7}}
\put(52.5,27){\vector(0,-3){8}}
\put(30,5){\vector(3,0){7}}
\put(73.5,5){\vector(-3,0){13}}
\put(50,6){\vector(0,3){8}}
\put(55,15){\vector(0,-3){8}}
\put(52.5,-4){\vector(0,3){7}}
\end{picture}
\end{center}

\paragraph*{}
\section{ISP-rings}
We shall begin with the following definition.
\begin{defn}
$A$ is said to be an ISP-ring if every proper regular ideal of $A$ can be factored as an invertible ideal times a nonempty product of proper radical ideals.
\end{defn}
Thus a domain is an ISP-ring if and only if it is an ISP-domain. Total quotient rings and SP-rings are clearly ISP-rings.
%\begin{rems}
%$(1)$ An integral domain is an ISP-ring if it is an ISP-domain.
%
%$(2)$ Recall from \cite{SP}, that an $SP$-ring is a ring in which every regular ideal is a product of radical ideals. Then, total quotient rings and $SP$-rings are clearly ISP-rings.
%
%$(3)$ A ring $A$ is ZPUI-ring if every proper ideal of $A$ can be written as a product of invertible and prime ideals, for instance, see\cite{O2}. Also, ZPUI-rings are other examples of ISP-rings.
%\end{rems}
Let $A$ be a ring and $I$ an ideal of $A$. Recall from \cite{Huckaba} that $I$ is an invertible ideal if and only if  $I$ is a finitely generated regular ideal and for each maximal ideal $M$ of $A$, $I_M$ is a principal ideal of $A_M$. We start by the following proposition in which we investigate the transfer of the ISP-ring property to finite direct product rings.

\begin{prop}\label{Pro}
Let $B$ be a finite direct product of some family of rings $(A_i)_{i=1,...,n}$. Then $B$ is an ISP-ring if and only if each $A_i$ is an ISP-ring.
\end{prop}
\proof The proof is done by induction on $n$ and it suffices to check it for $n=2$. Assume that $A_1$ and $A_2$ are ISP-rings. If $I$ is a proper regular ideal of $B=A_1\times A_2$, then $I=I_1\times I_2$, where $I_1$ and $I_2$ are regular ideals of $A_1$ and $A_2$ respectively. By assumption each $I_i$ can be factored as an invertible ideal times a nonempty product of proper radical ideals. On the other hand, it is well known that the direct product of invertible ideals is invertible. Hence $I$ can also be factored as an invertible ideal times a nonempty product of proper radical ideals. Therefore, $B$ is an ISP-ring. Conversely, assume that $B$ is an ISP-ring. Let $I$ (resp. $J$) be a proper regular ideal of $A_1$ (resp. $A_2$). Then $I\times A_2$ (resp. $A_1\times J$) can be factored as an invertible ideal times a nonempty product of proper radical ideals of $A_1\times A_2$, as desired.\qed\\

The following proposition examines the transfer of the ISP-ring property to homomorphic images.
\begin{prop}\label{p1}
Let $f:A\longrightarrow B$ be a surjective ring homomorphism such that $Ker(f)$ is a prime ideal of $A$.
Assume that every regular ideal of $B$ can be lifted via $f$ to a regular ideal of $A$. If $A$ is an ISP-ring, then so is $B$.
\end{prop}
\proof Let $I$ be a proper regular ideal of $B$. By our assumption, $f^{-1}(I)$ is a proper regular ideal of $A$. As $A$ is an ISP-ring, we can write 
%then $f^{-1}(I)$ can be factored as an invertible ideal times a nonempty product of proper radical ideals; say 
$f^{-1}(I)=JH_1\cdots H_n$ with $J$ an invertible ideal, $n\geq 1$ and all $H_i$'s are proper radical ideals. Then $I=f(J)f(H_1)\cdots f(H_n)$, where $f(J)$ is an invertible ideal and each $f(H_i)$ a proper radical ideal of $B$ since $J$ and each $H_i$ contains $Ker(f)$ properly. This completes the proof of the Proposition.\qed

\begin{rem}
Note that the condition ``$Ker(f)$ is a prime ideal" is necessary. Indeed, let $(A,M)$ be a non-Pr\"ufer local domain, $K=qf(A)$ and  $R=A\propto K$. By Corollary \ref{triv} and \cite[Theorem 5]{MT}, $R$ is not an ISP-ring. Now, let $E$ be a nonzero vector space over $\frac{R}{M\propto K}$ and $T=R\propto E$. Clearly, $T$ is an ISP-ring, but $\frac{T}{0\propto E}(\simeq R)$ is not an ISP-ring.
\end{rem}

%\begin{cor}
% Let $A$ be an  ISP-ring and $P$ a regular prime ideal of $A$. Then $A/P$ is an ISP-domain.
%\end{cor}
%\begin{proof}
%Take $B=A/P$ and $f=\pi$ a canonical map, then by Proposition \ref{p1}, if $A$ is an ISP-ring then so is $B$.

%\qed
%Assume that $A$ is an ISP-ring and $I$ a proper regular ideal of $A$. Let  $L\supseteq I$ be a proper ideal of $A$. Then $L$ is a regular ideal, so it can be written as $L=JH_1\cdots H_n$ with $J$ an invertible ideal, $n\geq 1$ and all $H_i$'s are  proper  radical ideals. Since all ideals $J,H_1,...,H_n$ contain $I$, we get $L/I=(J/I)(H_1/I)\cdots (H_n/I)$ with $J/I$ invertible and each  $H_i/I$ a proper radical ideal.
%\end{proof}

Let $P$ be a prime ideal of a ring $A$. Denote by $A_{(P)}=\{a/b\in T(A)~|~a\in A, b\in A-P \text{ and } b \text{ is regular} \}$ the regular localization of $A$ at $P$. Here $T(A)$ denotes the total quotient ring of $A$.
\begin{prop}\label{Loca}
Let $A$ be an ISP-ring and $S$ a multiplicatively closed set which consists only of regular elements. Then $A_S$ is an ISP-ring (in particular, $A_{(P)}$ is an ISP-ring, where $P$ is a prime ideal of $A$).
\end{prop}
\proof Note that if $I$ is a proper regular ideal of $A_S$, then $Q=I\cap A$ is a proper regular ideal of $A$. By assumption, $Q=JH_1\cdots H_n$ where $J$ is an invertible ideal, $n\geq 1$ and all $H_i$'s are proper radical ideals. Hence $I=QA_S=(JA_S)(H_1A_S)\cdots (H_nA_S)$ where $JA_S$ is an invertible ideal and each $H_iA_S$ a radical ideal. To finish the argument, we claim that $H_i\cap S=\emptyset$ for at least one $i$. If not, then $I=JA_S$. Therefore, $J\subseteq JA_S\cap A=I\cap A=Q\subseteq J$ and hence $J=JH_1\cdots H_n$, a contradiction since $J$ is invertible. 

\qed
%Assume that $L_iA_S=A_S$ for each $i=1,...,n$ then $H=KA_S$. Put $P=L_1\cdots L_n$. Therefore $J=KP$ and $K\subseteq KA_S\cap A=H\cap A=J=KP\subseteq K$ and hence $K=KP$. Then $P=A$ since $K$ is an invertible ideal, a contradiction.\qed\\

In general, localization and factor ring of an ISP-ring need not to be ISP-rings as shown in the following example.
\begin{exam}{\cite[Remark 2.3]{SP}}
Let $A=F[[x,y,z,v]]/(x^{2},xy,xz,xv)$, where $F$ is a field. Then $A$ is an ISP-ring but $A_{(x,y,z)A}$ and $A/xA$ are not. Indeed, clearly $A$ is a total quotient ring and hence $A$ is an ISP-ring. But $A_{(x,y,z)A}\simeq F[[y,z,v]]_{(y,z)}$ and $A/xA\simeq F[[y,z,v]]$ are not  ISP-rings, because both rings have dimension greater than one and Noetherian ISP-domains are Dedekind, cf. \cite[Corollary 4]{MT}.
\end{exam}

%\begin{prop}
%If $A$ is an ISP-ring then every regular maximal ideal of $A$ contains an invertible ideal.
%\end{prop}
%\proof Let $M$ be a regular maximal ideal of $A$. Let $x\in M$ be a regular element. As $A$ is an ISP-ring we get $xA=JK_1\cdots K_n$, where $J$ is an invertible ideal and each $K_i$ is a radical ideal. Then $J$ and each $K_i$ are an invertible and $M$ contains one of them.\qed\\

Our next result studies the possible transfer of the ISP-ring property between a ring $A$ and the trivial ring extension $A\propto E$. Set $S=A-(Z(A)\cup Z(E))$.
\begin{thm}\label{exten}
Let $A$ be a ring and $E$ an $A$-module such that $E=sE$ for every $s\in S$. Then $A\propto E$ is an ISP-ring if and only if every proper ideal of $A$ not disjoint to $S$ can be factored as an invertible ideal times a nonempty product of proper radical ideals.
\end{thm}
\proof
Assume that $A\propto E$ is an ISP-ring. Let $I$ be a proper ideal of $A$ such that $I\cap S\neq \emptyset$. Then $I\propto E$ is a proper regular ideal of $A\propto E$ and hence can be factored as an invertible ideal times a nonempty product of proper radical ideals.
%$I\propto E=HK_1\cdots K_n$ where $H$ is an invertible ideal and each $K_i$ is a radical ideals of $A\propto E$. 
By \cite[Theorem 3.9]{Anderson2} and \cite[Theorem 3.2(3)]{Anderson2}, $I\propto E=(J\propto E)(H_1\propto E)\cdots (H_n\propto E)$ with $J$ an invertible ideal of $A$, $n\geq 1$ and all $H_i$'s are proper radical ideals of $A$ . We get $I=JH_1\cdots H_n$. Conversely, let $L$ be a proper regular ideal of $A\propto E$. By \cite[Theorem 3.9]{Anderson2}, $L=I\propto E$ where $I$ is a proper regular ideal of $A$ such that $I\cap S\neq \emptyset$ since $E=sE$ for every $s\in S$. By our assumption, we can write $I=\widetilde{J}\widetilde{H}_1\cdots \widetilde{H}_n$ with $\widetilde{J}$ an invertible ideal of $A$,  $n\geq 1$ and all $\widetilde{H}_i$'s are proper radical ideals of $A$. Then $L=I\propto E=(\widetilde{J}\propto E)(\widetilde{H}_1\propto E)\cdots (\widetilde{H}_n\propto E)$, where $\widetilde{J}\propto E$ is invertible and each $\widetilde{H}_i\propto E$ (by definition) a proper radical ideal. Indeed, since $\widetilde{J}$ is an invertible ideal of $A$ and $\widetilde{J}\cap S\neq \emptyset$, it is easily verified that $\widetilde{J}\propto E$ is regular, finitely generated and locally principal, as desired.

\qed  

%Let $A$ be an integral domain with quotient field $K$, $E$ a $K$-vector space and $R=A\propto E$ the trivial ring extension of $A$ by $E$, 
Theorem \ref{exten} specializes to the following result.

\begin{cor}\label{triv}
Let $A$ be an integral domain  with quotient field $K$, $E$ a $K$-vector space,
and $R=A\propto E$. Then $R$ is an ISP-ring if and only if $A$ is an ISP-domain.
\end{cor}

\begin{exam}\label{exam}
Let $R=\Z\propto\Q$ be a trivial ring extension of $\Z$ by $\Z$-module $\Q$. By Corollary \ref{triv}, $R$ is an ISP-ring.
\end{exam}

Using \cite[Proposition 17]{MT}, we can construct nontrivial examples of ISP-rings.
\begin{exam}
Let $C$ be an SP-domain but not Dedekind, $M=qC$ a maximal principal ideal of $C$ and $D$ a discrete rank one valuation domain with quotient field $C/M$. Assume there exists a unit $p$ of $C$ such that $\phi(p)$ generates the maximal ideal of $D$, where $\phi: C\longrightarrow C/M$ is the canonical map. So, $A=\phi^{-1}(D)$ is an ISP-domain and let $K=qf(A)$. By Corollary \ref{triv}, $R=A\propto K$ is an ISP-ring.
\end{exam}

\begin{exam}
Let $A$ be a ring and $\Omega$ the set of proper regular ideals of $A$ which cannot be factored as an invertible ideal times a nonempty product of proper radical ideals. For each $I\in \Omega$ choose a maximal ideal $M_I$ containing $I$ and set $E=\oplus_{I\in \Omega}A/{M_I}$. Then $A\propto E$ is an ISP-ring. Indeed, set $S=A-(Z(A)\cup Z(E))$ and note that $Z(E)=\cup_{I\in \Omega}M_I$. If $s\in S$, then $s(A/{M_I})=A/{M_I}$ for every $I\in \Omega$, so $sE=E$. Now we apply Theorem \ref{exten} to get the required result.
\end{exam}

%Note that the class of ISP-rings contains total quotient rings.
%\begin{rem}
%We cannot obtain nontrivial examples of ISP-rings $A\bowtie I$ starting with a total quotient ring $A$. Indeed, if $A$ is a total quotient ring, then so is $A\bowtie I$. To see this, let $(a,b)\in A\bowtie I$ be a regular element, then $a,b$ are regular elements in $A$ (hence units) and $a-b\in I$. Let $aa'=1$ and $bb'=1$ with $a',b' \in A$. Then $a'-b'\in I$, so $(a',b')\in A\bowtie I$.
%\end{rem}

Now, we present our main result about the transfer of the ISP-ring property to amalgamated duplication of a ring along an ideal. To this end we denote by $Reg(A)$ the set of regular elements of $A$.
\begin{thm}\label{dup}
Let $A$ be a ring and $I$ an ideal of $A$. Then:
\begin{itemize}
\item[(1)] If $A\bowtie I$ is an ISP-ring then so is $A$.
\item[(2)] Assume that $I=aI$ for each $a\in Reg(A)$. Then $A\bowtie I$ is an ISP-ring if and only if so is $A$.
\end{itemize}
\end{thm}
The next lemmas prepare the way.
\begin{lem}\label{regu}
Let $A$ be a ring and $I$ an ideal of $A$. Then the following statements are equivalent:
\begin{itemize}
\item[(1)] Every regular ideal of $A\bowtie I$ has the form $H\bowtie I$, where $H$ is a regular ideal of $A$.
\item[(2)] $I=aI$ for every $a\in Reg(A)$.
\end{itemize}
\end{lem}

\proof
$(1)\Rightarrow (2)$. Let $i\in I$ and $a\in Reg(A)$. Clearly, $(a,a)$ is a regular element of $A\bowtie I$. By  hypothesis, the ideal generated by $(a,a)$ contains $0\times I$ and hence $(0,i)=(a,a)(0,j)=(0,aj)$ for some $j\in I$. Therefore, $i=aj$, as desired.\\
$(2)\Rightarrow (1)$. Let $J$ be a regular ideal of $A\bowtie I$ and $(a,a+i)$ a regular element of $J$. It follows from the general formula of $Z(A\bowtie I$) that $a$ and $a+i$ are regular elements of $A$. We aim to prove that $0\times I\subseteq J$. If $h\in I$, then  there exists $j\in I$ such that $h=(a+i)j$. Then $(0,h)=(0,j)(a,a+i)\in J$ and therefore $J$ has the form $H\bowtie I$ where $H$ is a regular ideal of $A$. This completes the proof.

\qed

\begin{lem}\label{inver}
Let $A$ be a ring and $I$, $J$ two ideals of $A$. If $J\bowtie I$ is invertible, then so is $J$.
\end{lem}
\proof Clearly, $J$ is a finitely generated regular ideal of $A$ since $J\bowtie I$ is an invertible ideal of $A\bowtie I$. Now, let $M$ be a maximal ideal of $A$. If $I\subseteq M$, then $(J\bowtie I)_{M\bowtie I}\cong J_{M}\bowtie I_{M}$ and hence $J_{M}$ is a principal ideal of $A_{M}$. If $I\nsubseteq M$, then $(J\bowtie I)_{M\bowtie I}\cong J_{M}$ and hence $J_{M}$ is a principal ideal of $A_{M}$, as desired.
\qed

\begin{lem}\label{fin}
Let $A$ be a ring and $I$, $J$ two ideals of $A$. If $J$ is regular finitely generated and $I=aI$ for each $a\in Reg(A)$, then $J\bowtie I$ is a finitely generated ideal of $A\bowtie I$.
\end{lem}
\proof Suppose that $J$ is finitely generated, say $J=\sum_{i=0}^{n}Aa_i$. Clearly $\sum_{i=0}^{n}(A\bowtie I)(a_i,a_i)\subseteq J\bowtie I$. Now, let $(a,a+j)\in J\bowtie I$, we have $(a,a+j)=(\sum_{i=0}^{n}\alpha_ia_i,\sum_{i=0}^{n}\alpha_ia_i+j)$. Since $I=aI$ for each $a\in Reg(A)$, we get $j=zk$, where $z$ is a regular element of $J$. Then $j=\sum_{i=0}^{n}\beta_ika_i$ and hence $(a,a+j)=\sum_{i=0}^{n}(\alpha_i,\alpha_i+\beta_ik)(a_i,a_i)\in \sum_{i=0}^{n}(A\bowtie I)(a_i,a_i)$. Therefore $J\bowtie I$ is finitely generated, as desired.

\qed

\textbf{Proof of Theorem \ref{dup}}

$(1)$ Note that if $Q$ is a regular ideal of $A$, then $Q\bowtie I$ is a regular ideal of $A\bowtie I$. Since $A\bowtie I$ is an ISP-ring, we can write $Q\bowtie I=JH_1\cdots H_n$ with $J$ an invertible ideal, $n\geq 1$ and all $H_i$'s are proper radical ideals. Further, as $J$ and each $H_i$ contains $0\times I$, therefore $J=\widetilde{J}\bowtie I$ and $H_i=\widetilde{H}_i\bowtie I$,  where $\widetilde{J}$ is an invertible ideal of $A$ by Lemma \ref{inver} and each $\widetilde{H}_i$ a proper radical ideal of $A$. Hence $Q=\widetilde{J}\widetilde{H}_1\cdots \widetilde{H}_n$.

$(2)$ By $(1)$ it suffices to prove that if $A$ is an ISP-ring then $A\bowtie I$ is an ISP-ring. Let $L$ be a proper regular ideal of $A\bowtie I$. By Lemma \ref{regu}, $L=Q\bowtie I$ where $Q$ is a proper regular ideal of $A$. As $A$ is an ISP-ring, we can write $Q=JH_1\cdots H_n$ with $J$ an invertible ideal, $n\geq 1$ and all $H_{i}$'s are radical ideals. Then $Q\bowtie I=(JH_1\cdots H_n)\bowtie I=(J\bowtie I)(H_1\bowtie I)\cdots (H_n\bowtie I)$ since $I=aI$ for each $a\in Reg(A)$. Clearly $J\bowtie I$ is regular, locally principal and finitely generated by Lemma \ref{fin} and hence $J\bowtie I$ is an invertible ideal. Also, each $H_i\bowtie I$ is a proper radical ideal, therefore $A\bowtie I$ is an ISP-ring.

\qed

By using the above results on trivial ring extension and amalgamated duplication, we give examples of ISP-rings.
\begin{exam}
Let $A=\Z\propto \Q$ be a trivial ring extension of $\Z$ by $\Z$-module $\Q$ and $I=0\propto \Q$ an ideal of $A$. Then $A\bowtie I$ is an ISP-ring. Indeed, by Example \ref{exam}, $A$ is an ISP-ring. Since $I=(n,q)I$ for each $(n,q)\in Reg(A)$, therefore $A\bowtie I$ is an ISP-ring by Theorem \ref{dup}.
\end{exam}

\begin{exam}
Let $A=\Z/8\Z$ be a ring and $I=2\Z/8\Z$ an ideal of $A$. By Theorem \ref{dup}, $A\bowtie I$ is an ISP-ring.
\end{exam}
We conclude the section by the following remark.

\begin{rem}
We cannot obtain nontrivial examples of ISP-rings $A\bowtie I$ starting with a total quotient ring $A$. Indeed, if $A$ is a total quotient ring, then so is $A\bowtie I$ (and hence it is an ISP-ring). To see this, if $(a,b)\in A\bowtie I$ is a regular element, then $a,b$ are regular elements in $A$ (hence units) and $a-b\in I$. Let $aa'=1$ and $bb'=1$ with $a',b' \in A$. Then $a'-b'\in I$, so $(a',b')\in A\bowtie I$.
\end{rem}

\section{Marot ISP-rings}
All the rings in this section are {\it Marot}, that is, their regular ideals are generated by regular elements (see \cite[p. 13]{Huckaba}).

\begin{prop}\label{regprime}
Let $A$ be an ISP-ring and $P\subset M$ regular prime ideals. Then $P\subseteq M^{2}A_{(M)}$.
\end{prop}
\proof By Proposition \ref{Loca}, we may assume that $A$ has only one regular maximal ideal $M$. Suppose that $P\nsubseteq M^{2}$. As $A$ is a Marot ring, we get $x\in M\setminus P$ a regular element of $M$. It is a consequence of both $A$ being an ISP-ring and $P\not\subseteq M^2$ that $(P,x^2)$ has to be a radical ideal. Moreover, $(P,x^2)=(P,x)$, since $(P,x^2)$ is a radical ideal, which gives a contradiction after modding out by $P$.

\qed

\begin{lem}\label{p=jp}
Let $A$ be a ring and $P$ a prime ideal of $A$ such that $P\subset I$ for some multiplication ideal $I$ of $A$. Then $P=IP$.
\end{lem}
\proof 

There is some ideal $J$ of $A$ such that $P=IJ$. Since $P$ is a prime ideal of $A$ and $I\not\subseteq P$, we infer that $J\subseteq P$, and hence $J=P$. Therefore, $P=IP$. 

%$P$ is regular and $P\subset I$, so $I^{-1}P\subset A$ is a regular ideal of $A$. As $A$ is an ISP-ring, we can write $I^{-1}P=JH_1\cdots H_n$ with $J$ an invertible ideal, $n\geq 1$ and all $H_i$'s are proper radical ideals. Therefore, $P=WI$ where $W=JH_1\cdots H_n$. Since $P$ is a prime ideal and $I\nsubseteq P$, we get $W=P$. Thus $P=IP$, as desired.
\qed

Before we go any further, let us recall \cite[Theorem 2.1]{Lu}.
\begin{thm} {\em (\cite[Theorem 2.1]{Lu})}\label{invert-prin}
If $A$ is a ring with a unique regular maximal ideal $M$, then each invertible ideal of $A$ is principle (and hence generated by a regular element).
\end{thm}

\begin{lem}\label{lemprin}
Let $A$ be an ISP-ring, $P\subset M$ regular prime ideals and $x\in M\setminus P$ a regular element of $A$ such that $M$ is minimal over $(P,x)$. Then $MA_{(M)}$ is principal and hence invertible.
\end{lem}
\proof By Proposition \ref{Loca}, we may assume that $M$ is the only regular maximal ideal of $A$. We show that $M$ is not idempotent. On the contrary assume that $M^2=M$. Note that $\sqrt{(P,x)}=M$ is the only proper radical ideal containing $(P,x)$. As $A$ is an ISP-ring and $M=M^2$. Combining these assumptions with Theorem \ref{invert-prin},  we get $(P,x)=yM$ for some regular element $y \in A$. Also $P\subset yM$ implies $y\not\in P$, otherwise $P=yA\subseteq yM$ which is impossible. Hence by Lemma \ref{p=jp}, $P=yP$. From $x\in yM$, we get $x=yz$ for some $z\in M$. Now $yM=(yP, yz)$ implies $M=(P,z)$, so $M/P$ is a nonzero principal idempotent maximal ideal of $A/P$, a contradiction. Thus $M$ is not idempotent and let us pick a regular element $\pi\in M-M^2$. By Proposition \ref{regprime}, $M$ is the only regular prime ideal containing $\pi$, so $\pi A=M$ because $A$ is an ISP-ring.

\qed

An ideal $I$ of a ring $A$ is called zero-dimensional if the factor ring $A/I$ is zero-dimensional.
\begin{lem}\label{zero-dim}
Any proper invertible radical ideal of an ISP-ring is zero-dimensional.
\end{lem}
\proof Let $A$ be an ISP-ring and $I$ a proper invertible radical ideal of $A$. On the contrary assume that $\dim(A/I)\geq 1$. Then there exist two prime ideals $P\subset M$ and a regular element $x\in M\setminus P$ such that $I\subseteq P$ and $M$ is minimal over $(P,x)$. By Lemma \ref{lemprin}, $MA_{(M)}$ is a principal invertible ideal. Changing $A$ by $A_{(M)}$, we may assume that $M$ is the only regular maximal ideal of $A$. Then $M=wA$ for some regular element $w\in A$. Let $y\in I$ be a regular element. As $I\subset M$, we get $y=bw$ for some $b\in A$. If $b \notin M $ then $I=M$ which is impossible, so $y=aw^2$ for some $a\in A$ and hence $aw\in \sqrt{I}=I$. Thus $I\subseteq wI$ (because $A$ is a Marot ring) and we get $A\subseteq M$, a contradiction.

\qed

It is notable that for a Marot ring $A$, the quotient ring $A_{(M)}$ and the large quotient ring $A_{[M]}$ coincide for each regular maximal ideal $M$ of $A$, for instance, see \cite[Theorem 7.6]{Huckaba}.
Recall that a ring $A$ is an $N$-ring if $A_{(M)}$ is a discrete rank one Manis valuation ring for each regular maximal ideal $M$ of $A$ (see \cite{Gf} and \cite{L}). A domain is an $N$-ring if and only if it is almost Dedekind. A ring $A$ is Pr\"ufer if every finitely generated regular ideal of $A$ is invertible. Our next result extends \cite[Corollary 4]{MT}.
\begin{thm}\label{ispN}
Any ISP-ring in which every regular prime ideal is maximal, is an $N$-ring.
\end{thm}
\proof Let $A$ be an ISP-ring such that every regular prime ideal of $A$ is maximal. We may assume that $A$ is not a total quotient ring. Let $M$ be a regular maximal ideal of $A$. By Proposition \ref{Loca}, $A_{(M)}$ is an ISP-ring. Changing $A$ by $A_{(M)}$, we may assume that $M$ is the only regular prime ideal of $A$. Let $x\in M$ be a regular element of $A$. Since $M$ is the only regular prime ideal of an ISP-ring $A$, therefore $xA=JM^n$ with an invertible ideal $J$ and $n\geq 1$. So $M$ is invertible and hence not idempotent. Let $H$ be any finitely generated proper regular ideal of $A$. As $A$ is an ISP-ring, we get $H=\widetilde{J}M^k$ for some invertible ideal $\widetilde{J}$ and $k\geq 1$. Therefore $H$ being a product of two invertible ideals is invertible and hence $A$ is a Pr\"ufer ring. Thus, by \cite[Theorem 3]{L} $A$ is an $N$-ring.

\qed

Recall that a regular-Noetherian ring is a ring whose regular ideals are finitely generated.
\begin{cor}\label{rnisp}
For a ring $A$ the following assertions are equivalent.
\begin{itemize}
\item[(1)] $A$ is a Dedekind ring.
\item[(2)] $A$ is a regular-Noetherian SP-ring.
\item[(3)] $A$ is a regular-Noetherian ISP-ring.
\end{itemize}

%The regular-Noetherian ISP-rings are exactly the Dedekind rings.
\end{cor}
\proof 
For the equivalence of $(1)$ and $(2)$, see \cite[Corollary 2.7]{SP}. $(2)$ implies $(3)$ is clear by definition. So we just require to prove $(3)$ implies $(1)$.
Let $A$ be a regular-Noetherian ISP-ring. By \cite[Theorem 17]{Gf}, a regular-Noetherian $N$-ring is Dedekind so, following Theorem \ref{ispN}, it suffices to show that every regular prime ideal of $A$ is maximal. Assume, to the contrary, that $P\subset M$ are regular prime ideals of $A$. By Proposition \ref{Loca}, we may assume that $M$ is the only regular maximal ideal of $A$. Let $x\in P$ be a regular element of $A$. As $A$ is an ISP-ring, we can write $xA=JH_1\cdots H_n$ with $J$ an invertible ideal, $n\geq 1$ and all $H_i$'s are proper radical ideals. Since $xA$ is generated by a regular element, therefore each radical ideal $H_i$ is invertible and hence, by Lemma \ref{zero-dim}, is zero-dimensional. As $M$ is the only regular maximal ideal, we get $xA=JM^n$ and hence $J\subseteq P$. By the same argument as above we get $J=J_1M^{n_1}$ for some invertible ideal $J_1\subseteq P$ and $n_1\geq 1$. Therefore, inductively we get a chain of invertible ideals $J\subseteq J_1\subseteq  \ldots \subseteq J_k\ldots$ with $J_{k-1}=J_{k}M^{n_{k}}$ and $n_k\geq 1$. As $A$ is regular-Noetherian, we get $J_k=J_kM^{n_k}$ for some $k\gg 0$. Thus $A=M^{n_k}$, a contradiction. This finishes the proof.

\qed

\section{Strongly ISP-rings}
In this section, we extend the notion of ISP-domain in the other way to the setting of arbitrary rings. We call $A$ a strongly ISP-ring if every proper ideal of $A$ can be factored as an invertible ideal times a nonempty product of proper radical ideals. Clearly, strongly ISP-domains are exactly ISP-domains. ZPUI and von Neumann regular rings are trivial examples of strongly ISP-rings. Within the frame of total quotient rings, ISP-rings are exactly SSP-rings as there is no proper regular ideal. Note that every strongly ISP-ring is an ISP-ring. The converse is not true in general, as the following example shows.
\begin{exam}\label{example}
Let $(A,M)$ be a local ring which is not reduced and $E$ a nonzero A-module such that $ME=0$. Then $A\propto E$ is an ISP-ring which is not strongly ISP. Indeed, clearly $A\propto E$ is a total quotient ring and hence an ISP-ring. Now, assume that $0\propto E=(J\propto E)(H_1\propto E)\cdots (H_n\propto E)$ with $J\propto E$ an invertible ideal, $n\geq 1$ and all $H_i\propto E$'s are proper radical ideals. If $n=1$ we get $JH_1=0$ and hence $H_1=0$ since $J$ is an invertible ideal of $A$, a contradiction. If $n>1$ we get $E=0$, again a contradiction.
\end{exam}

The Strongly ISP-ring property is stable under factor with prime ideal, fraction and finite direct product ring formations.
\begin{prop}\label{sisp}
The following assertions hold:
\begin{itemize}

\item[(1)] If $A$ is a strongly ISP-ring and $P$ a prime ideal of $A$, then $A/P$ is an ISP-domain.
\item[(2)] If $S$ is a multiplicatively closed set of a strongly ISP-ring $A$, then $A_S$ is a strongly ISP-ring.
\item[(3)] A finite direct product of some family of rings $(A_i)_{i=1,...,n}$ is a strongly ISP-ring if and only if each $A_i$ is a strongly ISP-ring.
\end{itemize}
\end{prop}
\begin{proof} $(1)$ Let $L\supset P$ be a proper ideal of $A$. As $A$ is a strongly ISP-ring, we can write $L=JH_1\cdots H_n$ with $J$ an invertible ideal and all $H_i$'s are proper radical ideals. Since all ideals $J, H_1, \ldots , H_n$ contain $P$, we get 
\begin{center}
$L/P=(J/P)(H_1/P)\cdots(H_n/P)$. 
\end{center}
It is easy to check locally that $J/P$ is an invertible ideal and each $H_i/P$ a proper radical ideal.

The assertions $(2)$ and $(3)$ are easy to check.

\qed
\end{proof}

%\begin{prop}
% Let $A$ be an  ISP-ring and $I$ a proper regular ideal of $A$. Then $A/I$ is a strongly ISP-ring. In particular if $I$ is a prime ideal of $A$, then $A/I$ is an ISP-domain.
%\end{prop}
%\begin{proof}
%Assume that $A$ is an ISP-ring and $I$ a proper regular ideal of $A$. Let  $L\supseteq I$ be a proper ideal of $A$. Then $L$ is a regular ideal, so it can be written as $L=JH_1\cdots H_n$ with $J$ an invertible ideal, $n\geq 1$ and all $H_i$'s are  proper  radical ideals. Since all ideals $J,H_1,...,H_n$ contain $I$, we get $L/I=(J/I)(H_1/I)\cdots (H_n/I)$ with $J/I$ invertible and each  $H_i/I$ a proper radical ideal.
%\end{proof}
%\begin{rem}
%A particular case of above proposition is also the consequence of Proposition \ref{p1}. Indeed, if we take $B=A/P$ for some prime ideal $P$ in $A$ and $f=\pi$ a canonical map, then by Proposition \ref{p1}, if $A$ is an ISP-ring then so is $B$.
%\end{rem}
The following string of three lemmas are the straightforward translation from domain case to any arbitrary ring. The proofs are quite similar as for the domain case.
\begin{lem}
Let $A$ be a strongly ISP-ring and $P\subset M$  nonzero prime ideals of $A$. Then $P_M\subseteq M^2A_M$.
\end{lem}
\proof By Proposition \ref{sisp}(2), we may assume that $A$ is local with maximal ideal $M$. Assume that $P\nsubseteq M^2$ and take $x\in M\setminus P$. Since $A$ is a strongly ISP-ring and $P\nsubseteq M^2$, therefore $(P,x^2)$ is a radical ideal. So $(P,x^2)=(P,x)$ which gives a contradiction after modding out by $P$.

\qed

\begin{lem}\label{slemprin}
Let $A$ be a strongly ISP-ring, $P\subset M$ prime ideals and $x\in M\setminus P$ such that $M$ is minimal over $(P,x)$. Then $MA_M$ is a principal ideal.
\end{lem}
\proof By Proposition \ref{sisp}(1), factor ring $A/P$ is an ISP-domain. So, changing $A$ by $A/P$, we may assume that $P=0$ and $M$ is minimal over $xA$. Now apply \cite[Lemma 7]{MT} to get the desired result.

\qed

\begin{lem}\label{irsisp}
Any proper invertible radical ideal of a strongly ISP-ring is zero-dimensional.
\end{lem}
\proof Let $A$ be a strongly ISP-ring and $I$ a proper invertible radical ideal of $A$. On the contrary assume that $\dim(A/I)\geq 1$. Then there exist two prime ideals $P\subset M$ and $x\in M-P$ such that $I\subseteq P$ and $M$ is minimal over $(P,x)$. By Lemma \ref{slemprin}, $MA_M$ is principal. Changing $A$ by $A_M$, we may assume that $A$ is local with maximal ideal $M$. Then $I=yA$ and $M=zA$ for some regular elements $y,z\in A$. As $I\subset M$, we get $y=az^2$ for some $a\in A$. So $az\in \sqrt{yA}=yA$ and hence $y=az^2\in yzA$. Thus $1\in zA=M$, a contradiction.

\qed

Recall that a ring $A$ is called special primary if Spec$(A)=\{M\}$ and each proper ideal of $A$ is a power of $M$.
Note that zero-dimensional rings are total quotient, that is, they have no non-unit regular element (\cite[p.10]{Huckaba}).
\begin{prop}\label{spr}
Let $A$ be a zero-dimensional local strongly ISP-ring with maximal ideal $M$. Then $A$ is special primary.
\end{prop}
\proof It suffices to remark that $A$ is an SSP-ring and $M$ is the only proper radical ideal of $A$.

\qed

Recall that an almost multiplication ring is a ring whose localizations at its prime ideals are discrete rank one valuation domains or special primary rings. The following result is an analogue of \cite[Corollary 4]{MT}.
\begin{thm}\label{sispamr}
Let $A$ be a strongly ISP-ring such that every nonzero prime ideal of $A$ is maximal. Then $A$ is almost multiplication.

\end{thm}
\proof
If $A$ is a one-dimensional ISP-domain then, by \cite[Corollary 4]{MT}, $A$ is almost Dedekind. Assume that dimension of $A$ is zero. As remarked above $A$ is total quotient, so $A$ is an SSP-ring. By Proposition \ref{sisp}(2), we may assume that $A$ is local with maximal ideal $M$ and by Proposition \ref{spr}, zero-dimensional local strongly ISP-rings are special primary, as desired.

\qed

%If $A$ is total quotient ring then $A$ is SSP-ring, and SSP-rings are almost multiplication rings, see \cite[Theorem 3.3]{SP}. So, assume that $M$ is regular and $x\in M$ is a regular element. Since $M$ is the only radical ideal of $A$, we get $xA=yM^k$ for some regular element $y\in A$ and $k\geq 1$, so $M$ is invertible and hence principal. Thus $A$ is special primary ring, as desired.\qed\\
Recall that a ring $A$ is ZPI if every proper ideal of $A$ is a product of prime ideals. The following result also extends \cite[Corollary 4]{MT}. 
\begin{cor}\label{nsisp}
For a ring $A$ the following assertions are equivalent.
\begin{itemize}
\item[(1)] $A$ is a ZPI-ring.
\item[(2)] $A$ is a Noetherian SSP-ring.
\item[(3)] $A$ is a Noetherian strongly ISP-ring.
\end{itemize}
%Let $A$ is a Noetherian strongly ISP-ring. Then $A$ is a ZPI-ring.
\end{cor}
\proof 
For the equivalence of $(1)$ and $(2)$, see \cite[Corollary 3.5]{SP}. $(2)$ implies $(3)$ is clear by definition. So we just require to prove $(3)$ implies $(1)$. Let $A$ be a Noetherian strongly ISP-ring. By \cite[Theorem 13]{M}, any Noetherian almost multiplication ring is ZPI-ring, so following Theorem \ref{sispamr}, it suffices to prove that every nonzero prime ideal of $A$ is maximal. Assume, to the contrary, that $P\subset M$ are nonzero prime ideals of $A$. By Proposition \ref{sisp}(2), we may assume that $A$ is local with maximal ideal $M$. Pick an element $x\in M\setminus P$. Shrinking $M$, we may assume that $M$ is minimal over $(P,x)$. By Lemma \ref{slemprin}, $M$ is principal, that is, $M=yA$ for some $y\in A$. As $P\subset M$, we get $P=yP$. Since $A$ is a Noetherian local ring, so $P=0$ by Nakayama's lemma, a contradiction. This finishes the proof.

\qed

Recall that a nonzero $A$-module $E$ is a multiplication module if each submodule
of $E$ has the form $IE$ for some ideal $I$ of $A$. Following \cite{Anderson2}, we call an ideal of $A\propto E$ homogeneous if it has the form $I\propto V$, where $I$ is an ideal
of $A$ and $V$ a submodule of $E$ such that $IE\subseteq V$. Our next result collects some useful fact.
\begin{prop}\label{Strong}
Let $A$ be a ring and $E$ an A-module.
\begin{itemize}

\item[(1)] If $A\propto E$ is a strongly ISP-ring, then so is $A$.

\item[(2)] If $A$ is a von Neumann regular ring and $E$ is a multiplication A-module, then $A\propto E$ is a strongly ISP-ring.

\item[(3)] If $A\propto E$ is a strongly ISP-ring and $E=sE$ for each $s\in S$, then $E$ is a multiplication module, where $S=A\backslash (Z(A)\cup Z(E))$ .

\end{itemize}
\end{prop}
\proof $(1)$ Straightforward.

$(2)$ By \cite[Proposition 3.6]{SP}, $A\propto E$ is an SSP-ring and hence $A\propto E$ is a strongly ISP-ring.

$(3)$ Assume that $A\propto E$ is a strongly ISP-ring and $V$ a submodule of $E$. By \cite[Theorem 3.2 (3)]{Anderson2} and \cite[Theorem 3.9]{Anderson2}, we can write $0\propto V=(J\propto E)(H_1\propto E)\cdots (H_n\propto E)$ with $J\propto E$ an invertible ideal, $n\geq 1$ and all $H_i\propto E$'s are proper radical ideals. Hence $0\propto V=(JH_1\cdots H_n)\propto QE$ for some ideal $Q$ of $A$. Then $V=QE$, as desired.

\qed

We get the following result, where $Supp(E)$ denotes the support of an $A$-module $E$.
\begin{prop}
Let $A\propto E$ be a strongly ISP-ring in which every prime ideal is maximal. Then for each maximal ideal $M\in Supp(E)$, $A_M$ is a field and $E_M\simeq A_M$.
\end{prop}
\proof Let $M\in Supp(E)$ be a maximal ideal. By Proposition \ref{sisp} and \cite[Theorem 4.1(2)]{Anderson2}, $(A\propto E)_{M\propto E}=A_M\propto E_M$ is a strongly ISP-ring whose every prime ideal is maximal. By Theorem \ref{sispamr}, $A_M\propto E_M$ is an almost multiplication ring. Hence $A_M\propto E_M$ is a discrete rank one valuation domain or a special primary ring. By the proof of \cite[Lemma 4.9]{Anderson2}, $A_M$ is a field and $E_M\simeq A_M$ since $E_M\neq 0$, as desired.

\qed

Recall that an $A$-module $E$ is simple if it has no proper nonzero submodule. Moreover, $E$ is called divisible if for every regular element $a\in A$ and for every $e\in E$ there exists  $e'\in E$ such that $e=ae'$. Equivalently, $E=aE$ for every regular element $a\in A$. Our next result gives necessary and sufficient conditions for particular trivial ring extension of $A$ by $E$ to be a strongly ISP-ring.
\begin{prop}\label{car}
Let $A$ be an integral domain and $E$ a divisible $A$-module. Then $A\propto E$ is a strongly ISP-ring if and only if $A$ is an ISP-domain and $E$ a simple $A$-module.
\end{prop}
\proof Assume that $A\propto E$ is a strongly ISP-ring. By Proposition \ref{Strong}(a), $A$ is a strongly ISP-domain (i.e. $A$ is an ISP-domain). Let $V$ be a nonzero submodule of $E$. Then we can write $0\propto V=(J\propto E)(H_1\propto E)\cdots (H_n\propto E)$ with $(J\propto E)$ an invertible ideal, $n\geq 1$ and all $H_i\propto E$'s are proper radical ideals. Since $E$ is divisible, therefore $0\propto V=(JH_1\cdots H_n)\propto E$ which gives $V=E$. Conversely, let $L$ be an ideal of $A\propto E$. By \cite[Corollary 3.4]{Anderson2}, $L=I\propto E$ for some ideal $I$ of $A$. If $I=0$, then by \cite[Theorem 3.2]{Anderson2} $0\propto E$ is a radical ideal of $A\propto E$. So, assume that $I\neq 0$. Since $A$ is a strongly ISP-ring, therefore $I=\widetilde{J}\widetilde{H}_1\cdots \widetilde{H}_n$, where $\widetilde{J}$ is an invertible ideal, $n\geq 1$ and each $\widetilde{H}_i$ a radical ideal. We get that $L$ can be factored as an invertible ideal times a nonempty product of proper radical ideals, as desired.

\qed

\begin{rem}
In general, $A\propto E$ need not be a strongly ISP-ring. Indeed, let $A$ be an ISP-domain, $K=qf(A)$ and $E$ a K-vector space such that $dim_{K}(E)>1$. By Proposition \ref{car}, $A\propto E$ is not a strongly ISP-ring.
\end{rem}

The following result studies the strongly ISP-ring property for amalgamated duplication ring $A\bowtie I$.
\begin{thm}\label{dupli}
Let $A$ be a ring and $I$ an ideal of $A$.
\begin{itemize}
\item[(1)] If $A\bowtie I$ is a strongly ISP-ring, then so is $A$.
\item[(2)] If  $I$ is a finitely generated idempotent ideal of $A$, then $A\bowtie I$ is a strongly ISP-ring if and only if so is $A$.
\end{itemize}
\end{thm}
\proof  $(1)$ Let $L$ be an ideal of $A$. As $A\bowtie I$ is a strongly ISP-ring, we can write $L\bowtie I=(J\bowtie I)(H_1\bowtie I)\cdots (H_n\bowtie I)$ with $J\bowtie I$ an invertible ideal, $n\geq 1$ and all $H_i\bowtie I$,s are proper radical ideals. We  get $L=JH_1\cdots H_n$. By Lemma \ref{inver}, $J$ is invertible. Also, each $H_i$ is a proper  radical ideal of $A$, as desired.

$(2)$ We only need to prove ``only if" part. Suppose that $A$ is a strongly ISP-ring. As $I$ is a finitely generated idempotent ideal, we get $I=Ae$ for some idempotent element $e\in A$. Then $A$ is isomorphic to $I\times C$, where $C=A(1-e)$. Since $A$ is a strongly ISP-ring, therefore $I$ and $C$ are also strongly ISP-rings, cf. [Proposition \ref{sisp}(3)]. So, $A\bowtie I\simeq (I\times C)\bowtie (I\times 0)\simeq I\times I\times C$ is a strongly ISP-ring, again cf. [Proposition \ref{sisp}(3)].

\qed

%Our next example illustrate part $(2)$ of Theorem \ref{dupli}.
%\begin{exam}
%Let $A$ be a von Neumann regular and $I$ be a finitely generated ideal of $A$ (i.e, $I$ is principal ideal generated by some idempotent element $e$). By Theorem \ref{dupli}, $A\bowtie I$ is a strongly ISP-ring since $A$ is von Neumann regular ring.
%\end{exam}

We conclude by giving an example of a ring $A$ that is a strongly ISP-ring while $A\bowtie I$ is not.

\begin{exam}
Let $F$ be a field, $A=F\propto F$ and $I=0\propto F$ an ideal of $A$. Then $A$ is a strongly ISP-ring, by Proposition \ref{Strong}. Notice that $A\bowtie I\simeq A\propto I$. Hence, by Example \ref{example}, $A\bowtie I$ is not a strongly ISP-ring since $A$ is not a reduced ring.

\end{exam}
 
{\bf Acknowledgements.} We are grateful to Tiberiu Dumitrescu for his helpful comments on the initial draft of the paper. We also thank the referee whose comments and suggestions improved our paper.

%%%%%%%%%%%%%%%%%%%%%%%%%%%%%%%%%%%%%%%%%%%%%%%%%%%%%%%%
\end{document}